\documentclass[11pt,a4paper]{amsart}
\usepackage[T1]{fontenc}
\usepackage[latin9]{inputenc}
\usepackage{amssymb}
\usepackage{color}
\usepackage{graphicx,color}
\usepackage{amsmath, amssymb, graphics}

\makeatletter
\numberwithin{equation}{section} 
\numberwithin{figure}{section} 
  \@ifundefined{theoremstyle}{\usepackage{amsthm}}{}
  \theoremstyle{plain}
  \newtheorem{thm}{Theorem}[section]
  \theoremstyle{plain}
  \newtheorem{cor}[thm]{Corollary}
  \theoremstyle{plain}
  \newtheorem{prop}[thm]{Proposition}
  \theoremstyle{remark}
  \newtheorem{rem}[thm]{Remark}
  \theoremstyle{remark}
 \newtheorem{mydef}{Definition}
  
  \theoremstyle{plain}
  
  \theoremstyle{definition}
\newtheorem{exmp}{Example}[section]

\newcommand{\R}{\mathbb{R}}

\begin{document}

\title[A family of Solutions of the n Body problem]{A family of Solutions of the n Body problem}
\author{ Oscar Perdomo}
\date{\today}

\curraddr{Department of Mathematics\\
Central Connecticut State University\\
New Britain, CT 06050\\}

\email{ perdomoosm@ccsu.edu}

\begin{abstract}

In this paper we characterize all the solutions of the three body problem on which one body with mass $m_1$ remains in a fixed line and the other two bodies have the same mass $m_2$. We show that all the solutions with negative  total energy (potential plus kinetic), never collide.   ``Explicit'' solutions of the motion of all the three bodies are given in term of just one solution of an ordinary differential equation of order two of one real value function. We show that the solutions of a  big subfamily of this family remains bounded for all time. The same results are shown for the n-body problem. Once we have shown that there are infinitely many periodic solutions, we provide several examples of periodic examples for the 3-Body, 4-Body, 5-Body and 6-Body problem. This is a first version. Comments and suggestions for references are welcome. A YouTube video explaining one of this solutions has been posted on  http://youtu.be/2Wpv6vpOxXk

\end{abstract}

\maketitle

\section{Introduction}

\begin{figure}[hbtp]
\begin{center}
\includegraphics[width=.4\textwidth]{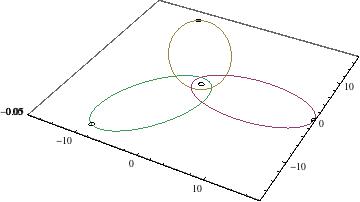}
\end{center}
\caption{The periodic solutions that we are considering have one body of mass $m_1$ moving up and down along a line. The other $n-1$ bodies have mass $m_2$, and move around. There are periodic solutions with $m_1>m_2$ and with $m_2>m_1$. In the solution shown above, $m_1=368.5$ and $m_2=484$. No units are given since we are assuming that the Gravitational constant is 1.  }
\end{figure}

Let us assume that $x(t)=(x_1(t),x_2(t),x_3(t))$, $y(t)=(y_1(t),y_2(t),y_3(t))$ and $z(t)=(z_1(t),z_2(t),z_3(t))$ describe the motion of three bodies with masses $m_1$, $m_2$ and $m_2$ respectively that move under the influence of the force of gravity. If we assume that the Gravitational constant is $1$, then, the functions $x$, $y$ and $z$ satisfy the equations

\begin{eqnarray*}
\ddot{x}(t)&=&\frac{m_2}{|y-x|^3}\, (y-x)+ \frac{m_2}{|z-x|^3}\, (z-x)\\
\ddot{y}(t)&=&\frac{m_1}{|x-y|^3}\, (x-y)+ \frac{m_2}{|z-y|^3}\, (z-y)\\
\ddot{z}(t)&=&\frac{m_1}{|x-z|^3}\, (x-z)+ \frac{m_2}{|y-z|^3}\, (y-z)
\end{eqnarray*}

We will study the case when the body $x(t)$ moves in the fixed line $l=\{(0,0,u):u\in \R\}$ and since our intention is not to study the three dimensional problem in a line (when all the bodies lies in a single line), then we will assume that at $t=0$, one of the other planets is not in the line $l$. Since we are assuming that  $x(t)$ lies in $l$, then we will write $x(t)=(0,0,f(t))$ for some function $f$. In order to make the study easier, we will assume that $f(0)=0$ and also we will assume that the center of mass does not move, this is, we will assume that $m_1x(t)+m_2y(t)+m_2z(t)=(0,0,0)$ for all $t$.  We proceed by noticing that:

\begin{center}{\it With the additional condition that the center of mass does not move, we have that the projections of the bodies $y$ and $z$ in the plane $\Pi=\{(u,v,0):u,v\in \R \}$ given by $Y=(y_1(t),y_2(t))$ and $Z=(z_1(t),z_2(t))$ satisfy that $Z=-Y$ and the projections of the bodies $y$ and $z$ on $l$ satisfy that $y_3=z_3=-\frac{m_1}{2m_2} f(t)$ . In particular there is a collision only when $r=|Y|=\sqrt{y_1(t)^2+y_2(t)^2}$ vanishes.
}
\end{center}

We then point out that not only the angular momentum of all the system remains constant but the angular momentum of each body remains constant. We have that 

\begin{center}{\it 
$ c_1=y_1(t)\dot{y}_2(t)-y_2(t)\dot{y}_1(t)$ is constant  
}
\end{center}

Besides the constant $c_1$ defined above, we will consider the constant $c_2$ given by the conservation of the total energy, this is 

\begin{center}{\it 
$$ c_2=-\frac{m_1m_2}{|x-y|}-\frac{m_1m_2}{|x-z|}-\frac{m_2^2}{|z-y|}+\frac{1}{2} \left(m_1 |\dot{x}|^2+m_2 |\dot{y}|^2+m_2 |\dot{z}|^2\right) $$
}\end{center}
In terms of $c_1$ and $c_2$ we show that

\begin{center}{\it  If $c_2<0$ and $c_1\ne 0$, then $0<a_1 <r(t)< b_1$ for some explicit expression $a_1$ and $b_1$ in terms of $c_1,c_2,m_1$ and $m_2$. See Theorem \ref{t2} }
\end{center}
Moreover, we prove that a big part of this family is made out of bounded solutions. More precisely we show:

\begin{center}{\it  If $c_2<0$, $c_1\ne 0$ and $8 m_1m_2^2+m_2^3+16 c_2c_1^2 >0$, then the solution is bounded}
\end{center}

Once we have shown that an open set of the whole family under consideration corresponds to bounded solutions, by the continuity of the solutions of an ordinary differential equation with respect to the initial conditions, we have that:

\begin{center}{\it there are infinitely many periodic solutions.}
\end{center}

With regards of finding explicit formulas for the solutions of $x$, $y$ and $z$,  we prove the following:

\begin{center}{\it  As in the two body problem we have that if $Y=(y_1,y_2)=r(\cos \theta,\sin\theta)$ then we have that $r^2(t)\dot{\theta}(t)=c_1$. Therefore, if we know $r(t)$, we can compute $\theta(t)$ by doing one integral. We prove that on any interval where $\dot{r}(t)$ does not vanish, if $f(t)=g(r(t))$ and $k=\frac{m_1+2m_2}{2 m_2}$, then $g$ satisfies the differential equation of the form
$
a(r,g,g^\prime)\,  g^{\prime\prime}(r)+b(r,g)\, g^\prime+c(g,r)=0\, .\, 
$
See Theorem \ref{t3} }
\end{center}

If we assume that we know the function $g$, we can find the function $r(t)$  by solving, using  separation of variables, the differential equation:

\begin{eqnarray*}
-\frac{8 (-1+k) m_2}{\sqrt{r^2+k^2g(r)^2}}+\frac{2 c_1^2-m_2 r+2 r^2 \left(1+(-1+k) k g'(r)^2\right) \dot{r}^2}{r^2}=\frac{2 c_2}{m_2}
\end{eqnarray*}

Therefore on any interval where $\dot{r}(t)$ does not vanish, we have explicit formulas for $x(t)$, $y(t)$ and $z(t)$ in terms of the function $g$ and antiderivatives of expressions of function $g$. Let us see this: once we have $g(r)$ we can find $r(t)$, then we can find $f(t)$ because $f(t)=g(r(t))$ and as mentioned earlier, we can find $\theta(t)$ using $r^2(t)\dot{\theta}(t)=c_1$. We then conclude that,

\begin{eqnarray*}
x(t)&=&(0,0,f(t))=(0,0,g(r(t)))\\
y(t)&=&(r(t) \cos(\theta(t)),r(t) \sin(\theta(t)),-\frac{m_1}{2m_2} f(t))\\
z(t)&=&(-r(t) \cos(\theta(t)),-r(t) \sin(\theta(t)),-\frac{m_1}{2m_2} f(t))
\end{eqnarray*}

We have the same type of results for the $n$-body problem. In this first version of this paper we will sketch the proofs and present the images of several of the periodic examples.

\section{The results for the 3-Body Problem}
The following theorem characterizes all the solutions of the problem that we are considering.

\begin{thm}\label{t1}
Let  $x=(0,0,f(t))$, $y(t)=(y_1(t),y_2(t),y_3(t))$ and $z(t)=(z_1(t),z_2(t),z_3(t))$ and $r(t)=\sqrt{y_1(t)^2+y_2(t)^2}$. $x$, $y$ and $z$
describe a solution of the Three Body Problem with masses $m_1$, $m_2$ and $m_2$ respectively that satisfies $r(0)\ne0$  and
\begin{eqnarray}\label{cm}
m_1x(t)+m_2y(t)+m_2z(t)=(0,0,0) 
\end{eqnarray}

if and only if 

\begin{eqnarray*}
y(t)&=&(r(t) \cos(\theta(t)),r(t) \sin(\theta(t)),-\frac{m_1}{2m_2} f(t))\\
z(t)&=&(-r(t) \cos(\theta(t)),-r(t) \sin(\theta(t)),-\frac{m_1}{2m_2} f(t))
\end{eqnarray*}

with

\begin{eqnarray}\label{e1}
\ddot{f}=\frac{-2m_2k f}{h^3}\quad\hbox{and}\quad \ddot{r}=\frac{c_1^2}{r^3}-\frac{m_2}{4 r^2}-2(k-1)m_2\frac{r}{h^3}
\quad\hbox{and}\quad r^2\dot{\theta}=c_1
\end{eqnarray}\label{ode1}
where $k=\frac{m_1+2m_2}{2m_2}$, $h=\sqrt{r^2+k^2f^2}$ and $c_1=y_1(0)\, \dot{y_2}(0)-y_2(0)\, \dot{y_1}(0)$
\end{thm}

\begin{proof}
Let us prove one implication first. Let us assume that $x,y $ and $z$ satisfy the 3-Body ODE

\begin{eqnarray*}
\ddot{x}(t)&=&\frac{m_2}{|y-x|^3}\, (y-x)+ \frac{m_2}{|z-x|^3}\, (z-x)\\
\ddot{y}(t)&=&\frac{m_1}{|x-y|^3}\, (x-y)+ \frac{m_2}{|z-y|^3}\, (z-y)\\
\ddot{z}(t)&=&\frac{m_1}{|x-z|^3}\, (x-z)+ \frac{m_2}{|y-z|^3}\, (y-z)
\end{eqnarray*}

Let us denote by $Y=(y_1(t),y_2(t))$ and $Z(t)=(z_1(t),z_2(t))$. From Equation (\ref{cm}) we obtain that

\begin{eqnarray}\label{e1}
z_3=\frac{-m_1 f-m_2 y_3}{m_2}\quad\hbox{and}\quad Z=- Y
\end{eqnarray}

Using the first equation of the 3-Body ODE, we conclude that,

\begin{eqnarray}\label{e2}
\frac{m_2 Z}{|z-x|^3}=\frac{-m_2Y}{|y-x|^3}
\end{eqnarray}

Combining  equation (\ref{e1}) with equation (\ref{e2}) and the fact that $r(0)\ne0$, we conclude that $|z-x|=|y-x|$. A direct computation shows that 

\begin{eqnarray}\label{e3}
|z-x|^2=\frac{m_2 r^2+((m_2+m_1)f+m_2y_3)^2}{m_2^2}
\end{eqnarray}

\begin{eqnarray}\label{e3}
|y-x|^2=r^2+(f-y_3)^2
\end{eqnarray}

Since $|y-x|=|z-x|$ we conclude from the two expressions above that either $(m_2+m_1)f+m2y_3=m_2(f-y_3)$ and then 

\begin{eqnarray}\label{me3}
y_3=-\frac{m_1}{2 m_2}\, f
\end{eqnarray} 

or $(m_2+m_1)f+m2y_3=-m_2(f-y_3)$ and therefore $f=0$. When $f=0$, we conclude that $y=-z$ and that the motion of $y$ takes place in a planar conic. Up to rigid motion of $\R^3$, this case is covered in the solution presented in the theorem. For this reason we will assume  that Equation (\ref{e3}) holds. Notice that Equations (\ref{e1}) and (\ref{me3}) imply that 

\begin{eqnarray}\label{e4}
z_3=-\frac{m_1}{2 m_3}\, f
\end{eqnarray}

Using the first equation of the 3-Body Problem and Equations  (\ref{e3}) and (\ref{me3}) we conclude that 

\begin{eqnarray}\label{eqforf}
\ddot{f}=\frac{-2m_2k f}{h^3}\quad\hbox{where}\quad  k=\frac{m_1+2m_2}{2m_2},\,  \hbox{and}\, h=\sqrt{r^2+k^2f^2}
\end{eqnarray}

From the second equation of the 3-Body Problem we conclude that the function $Y$ satisfies the differential equation, 

\begin{eqnarray}\label{eqforY}
\ddot{Y}= -\left( \frac{m_1}{h^3}+\frac{m_2}{r^3}\right) Y
\end{eqnarray}

If we write $Y=(r\cos\theta,r\sin\theta)$ then we get that $\ddot{Y}=(\ddot{r}-r\dot{\theta}^2)\, (r\cos\theta,r\sin\theta)+(2 \dot{r}\dot{\theta}+r\ddot{\theta})\, (-r\sin\theta,r\cos\theta)$. Comparing this equation with Equation (\ref{eqforY}) we conclude that $\ddot{r}=\frac{c_1^2}{r^3}-\frac{m_2}{4 r^2}-2(k-1)m_2\frac{r}{h^3}
\quad\hbox{and}\quad r^2\dot{\theta}=c_1
$ where $c_1=y_1(0)\, \dot{y_2}(0)-y_2(0)\, \dot{y_1}(0)$. This finishes one implication of the theorem. The converse is a direct computation.

\end{proof}

\begin{cor}
When the function $f$ is identically zero, the body with mass $m_1$ will not move and the other two bodies will move on a conic.  For the bounded solutions the bodies move along ellipses and are periodic. The motion of each solution is essentially that of a binary star system for the bodies moving around while the body with mass $m_1$ stays at the center of mass.
\end{cor}

The following theorem shows that when the total energy is negative, then, there is not collision.

\begin{thm}\label{t2}
If a solution of the 3-Body problem considered in Theorem \ref{t1} has nonzero angular momentum, this is $c_1\ne0$, and it also has negative total energy, this is: 

\begin{eqnarray}\label{te1}
c_2=-2\frac{m_1m_2}{h}-\frac{m_2^2}{2 r}+m_2 \dot{Y}\cdot\dot{Y} + \frac{m_1^2+2 m_1m_2}{4m_2} \dot{f}^2 <0
\end{eqnarray}

then, there is no collision. Moreover

\begin{eqnarray} \label{dDelta}
D=16 c_1^2 c_2 m_2+(4 m_1m_2+m_2^2)^2>0
\end{eqnarray}

and 

\begin{eqnarray} \label{rb}
\frac{-4m_1m_2-m_2^2+\sqrt{D}}{4c_2}<r(t)< \frac{-4m_1m_2-m_2^2-\sqrt{D}}{4c_2}
\end{eqnarray}

\end{thm}

\begin{proof}

A direct computation shows that $\dot{Y}=\dot{r} (\cos\theta,\sin\theta)+r\dot{\theta}(-\sin\theta,\cos\theta)$. It follows that 
$\dot{Y}\cdot\dot{Y}= \dot{r}^2+r^2(\theta^\prime)^2=\dot{r}^2+\frac{c_1^2}{r^2}$ and therefore Equation (\ref{te1} )) transforms into

\begin{eqnarray}\label{te1}
c_2=m_2\frac{c_1^2}{r^2} - 2\frac{m_1m_2}{h} - \frac{m_2^2}{2r}+\frac{m_1^2+2 m_1m_2}{4m_2} \dot{f}^2 +  m_2\dot{r}^2
\end{eqnarray}

From Equation (\ref{te1}) we get that $m_2\frac{c_1^2}{r^2} - 2\frac{m_1m_2}{h} - \frac{m_2^2}{2r}<c_2$ and therefore

$$-c_2 r^2<2m_1m_2 \frac{r^2}{h}+\frac{m_2^2}{2}\, r -m_2c_1^2<2m_1m_2 r+\frac{m_2^2}{2}\, r -m_2c_1^2$$

The result follows because we know that $r$ is positive and lies in the domain where the parabola $u \to c_2u^2+2m_1m_2 u+\frac{m_2^2}{2}\, u -m_2c_1^2$ is positive.

\end{proof}

\begin{cor}
If $c_2<0$ and $c_1\ne0$, then the solution is defined for all time. 
\end{cor}

\begin{mydef}\label{fam1}
We will denote by 
\begin{eqnarray*}
\mathbb{L}&=&\{(x,y,z)\in \R^9: x,y,z \, \hbox{is a solution of the 3-Body Problem satisfying }\\
    & &\quad x(0)=(0,0,0),\,  y(0)=(y_{10},0,0)\, z(0)=(-y_{10},0,0) , \, c_2<0,\, c_1\ne 0\\
    & &\quad x^\prime(0)=(0,0,df_0),\,  y^\prime(0)=(0,dy_{20},-\frac{m_1df_0}{2 m_2}),\, z^\prime(0)=(0,-dy_{20},-\frac{m_1df_0}{2 m_2})\}\\
\end{eqnarray*}
where $c_1=y_{10} dy_{20}$ and $c_2=m_2 dy_{20}^2-\frac{m_2}{2y_{10}} (4m_1+m_2)+\frac{m_1}{4 m_2} (m_1+2m_2) df_0$
\end{mydef}
\begin{rem}\label{me}
The space $\mathbb{L}$ contains solutions of the three body problem that do not collide because $r$ is bounded. Also they are defined for all $t$. The initial conditions were taken so that the body $x$ starts at the origin and the derivative of $r$ at $t=0$ is zero. Notice that $\mathbb{L}$ can be parametrized by the initial conditions and the value of the masses, this it, it can be parametrized by points in $\R^5$ of the form $q=(y_{10},dy_{20},df_0,m_1,m_2)$. We will denote by $\mathbb{L}(q)$ the solution that satisfies the initial conditions  given by the point $q$.

\end{rem}

\begin{thm}\label{bs3BP}
The solutions of the 3-Body Problem considered in Theorem \ref{t1} with $c_1\ne0$, $c_2<0$ and $8 m_1m_2^2+m_2^3+16 c_2c_1^2 <0$ are bounded
\end{thm}
\begin{proof}
Since $c_2<0$ and $c_1\ne0$, then $r$ is bounded. Since $r$ satisfies the second order differential Equation \ref{ode1}, then  $r$ has not horizontal asymptotes,  and since $r$ is bounded, then $\ddot{r}$ changes sign infinitely many times. We can also show that if  $f$ is unbounded then $f(t)$ converges to either infinity or negative infinity. This is a contraction because $\ddot{r}=0$, $c_2<0$ and $f$ large imply that $8 m_1m_2^2+m_2^3+16 c_2c_1^2 \ge 0$

\begin{mydef}\label{fam1}
We will denote by 
\begin{eqnarray*}
\mathbb{B}&=&\{(x,y,z)\in \mathbb{L}:  8 m_1m_2^2+m_2^3+16 c_2c_1^2 <0,\}
\end{eqnarray*}
where $c_1=y_{10} dy_{20}$ and $c_2=m_2 dy_{20}^2-\frac{m_2}{2y_{10}} (4m_1+m_2)+\frac{m_1}{4 m_2} (m_1+2m_2) df_0$

\end{mydef}

\end{proof}

\begin{prop}
For a given $\theta_0$ of the form $a \pi$ where  $a$ is a rational number, let us consider

\begin{eqnarray*}\label{xi}
\xi(y_{10},dy_{20},df_0,m_1,m_2,t_0)= (f^\prime(t_0)-df_0)^2+(\theta(t_0)-\theta_0)^2+r^\prime(t_0)^2+f(t_0)^2 
\end{eqnarray*}

We have that a solution in $\mathbb{L}$ is periodic with period $t_0$ if and only if $\xi(y_{10},dy_{20},df_0,m_1,m_2,t_0)=0$

\end{prop}

\begin{proof}
It follows by the existence and uniqueness theorem of ordinary differential equations.
\end{proof}

\begin{thm}
There are infinitely many periodic non trivial solution of the 3BP in the family $\mathbb{B}$.
\end{thm}

\begin{proof}
It follows by the continuous dependence on the initial condition of the ordinary differential equation \ref{ode1}. 
\end{proof}

\begin{thm}\label{t3}
For every solution of the 3-Body problem considered in Theorem \ref{t1}, we have that 
\begin{eqnarray}\label{odefg}
a(r,g,g^\prime)\,  g^{\prime\prime}(r)+b(r,g)\, g^\prime+c(g,r)=0
\end{eqnarray}

where

\begin{eqnarray*}
a(r,g,g^\prime)&=&\frac{\left(8 (1-k) m_2^2 r^2+\left(2 c_1^2 m_2-r \left(m_2^2+2 c_2 r\right)\right) \sqrt{r^2+k^2 g(r)^2}\right)}{2 m_2 r^2 \sqrt{r^2+k^2 g(r)^2} \left(1+(k-1) k g'(r)^2\right)} \\
b(r,g)&=&-\left(\frac{c_1^2}{r^3}-\frac{m_2}{4 r^2}-\frac{2 (-1+k) m_2 r}{\left(r^2+k^2 g(r)^2\right)^{3/2}}\right) \\
c(r,g)&=& -\frac{2 k m_2 g(r)}{\left(r^2+k^2 g(r)^2\right)^{3/2}} 
\end{eqnarray*}

\end{thm}

\begin{proof}
We will show if $f(t)=g(r(t))$, then $g$ satisfies the Equation (\ref{odefg}). Once we have shown this,  the converse statement, which is the theorem here, would be a direct computation.   A direct computation using Equation (\ref{e1}) shows us that the Equation $0=\frac{d^2(f(t)-g(r(t))}{dt^2}$ is equivalent to  
\begin{eqnarray}\label{tempe1}
c(r,g)+b(r,g) g^\prime(r)-(\dot{r})^2\, g^{\prime\prime}=0
\end{eqnarray}
Replacing $\dot{f}$ by $\dot{r}g^\prime$ in Equation (\ref{te1}) allows to solve $\dot{r}^2$ in terms of $r, \, g$ and $g^\prime$. Replacing this expression for $\dot{r}^2$ in Equation (\ref{tempe1}) completes the proof of this theorem.

\end{proof}

\section{The results for the 4-Body Problem}


\begin{thm}\label{t4}
Let  $x=(0,0,f(t))$, $y(t)=(y_1(t),y_2(t),y_3(t))$, $z(t)=(z_1(t),z_2(t),z_3(t))$ and $w(t)=(w_1(t),w_2(t),w_3(t))$ and $r(t)=\sqrt{y_1(t)^2+y_2(t)^2}$. $x$, $y$ $z$ and $w$
describe a solution of the Four Body Problem with masses $m_1$, $m_2$, $m_2$ and $m_2$ respectively that satisfy $\{y(0),z(0),w(0)\}$ are linearly independent,  $r(0)\ne0$   and
\begin{eqnarray}\label{cm}
m_1x(t)+m_2y(t)+m_2z(t)+m_2w(t)=(0,0,0) 
\end{eqnarray}

if and only if 

\begin{eqnarray*}
x(t)&=&\left(0,0,f(t)\right)\\
y(t)&=&\left(r(t)\cos(\theta(t))\, ,\, r(t)\sin(\theta(t))\, ,\, -\frac{m_1}{3 m_2}\, f(t)\right)\\
z(t)&=&\left(r(t)\cos(\theta(t)+\frac{2 \pi}{3})\, ,\, r(t)\sin(\theta(t)+\frac{2 \pi}{3})\, ,\, -\frac{m_1}{3 m_2}\, f(t)\right)\\
w(t)&=&\left(r(t)\cos(\theta(t)+\frac{4 \pi}{3})\, ,\, r(t)\sin(\theta(t)+\frac{4\pi}{3})\, ,\, -\frac{m_1}{3m_2}\, f(t)\right)\\
\end{eqnarray*}

with

\begin{eqnarray}\label{e2}
\ddot{f}&=&-\frac{(m_1+3m_2)f}{h^3}\\
 \ddot{r} &=& \frac{c_1^2}{r^3}-\frac{3 m_2}{l^3 r^2}-\frac{m_1r}{h^3} \quad\hbox{and}\quad r^2\dot{\theta}=c_1
\end{eqnarray}

where  $h=\sqrt{r^2+(\frac{m_1+3m_2}{3m_2})^2f^2}$, $l= \sqrt{3}$ and $c_1=y_1(0)\, \dot{y_2}(0)-y_2(0)\, \dot{y_1}(0)$

\end{thm}

\begin{proof}
This is a direct computation similar to those made on the proof of Theorem (\ref{t1}). 
\end{proof}

\begin{cor}
When the function $f$ is identically zero, the body with mass $m_1$ will not move and the other two bodies will move on a conic.  For the bounded solutions the bodies move along ellipses and are periodic. The motion of each solution is essentially that of a Trinary start system for the bodies moving around and the body with mass $m_1$ stays at the center of mass.
\end{cor}

\begin{thm}\label{t5}
If a solution of the 4-Body problem considered in in Theorem \ref{t4} has nonzero angular momentum, this is $c_1\ne0$, and it also has negative total energy, then, there is no collision. Moreover

\begin{eqnarray} \label{dDelta}
D=6 c_1^2 c_2 m_2+(3 m_1m_2+\frac{3}{l} m_2^2)^2>0
\end{eqnarray}

and 

\begin{eqnarray} \label{rb}
\frac{-3m_1m_2-\frac{3}{l} m_2^2+\sqrt{D}}{2c_2}<r(t)<\frac{-3m_1m_2-\frac{3}{l} m_2^2-\sqrt{D}}{2c_2}
\end{eqnarray}

\end{thm}

\begin{proof}
This is a direct computation similar to those made on the proof of Theorem (\ref{t2}). 
\end{proof}

\begin{thm}\label{bs4BP}
The solutions of the 4-Body Problem considered in Theorem \ref{t4} with $c_1\ne0$, $c_2<0$ and $2 \sqrt{3}  m_1m_2^2+m_2^3+2 c_2c_1^2 <0$ are bounded
\end{thm}
\begin{proof}
This is a direct computation similar to those made on the proof of Theorem (\ref{bs3BP}). 
\end{proof}

\begin{cor}
There are infinitely many periodic solutions considered in Theorem \ref{t4}
\end{cor}

\begin{thm}\label{t6}
For every solution of the 3-Body problem considered in Theorem \ref{t1}, we have that 
\begin{eqnarray}\label{odefg4BP}
a(r,g,g^\prime)\,  g^{\prime\prime}(r)+b(r,g)\, g^\prime+c(g,r)=0
\end{eqnarray}

where $k=\frac{m_1+3m_2}{3m_2} $and

\begin{eqnarray*}
a(r,g,g^\prime)&=&-\frac{18 l m_1 m_2^2 r^2+\left( 18 m_2^3r+6lc_2m_2r^2-9lc_1^2m_2^2\right) \sqrt{r^2+k^2 g(r)^2}}{l  r^2 \sqrt{r^2+k^2 g(r)^2} \left(9m_2^2+m_1(m_1+3m_2) g'(r)^2\right)} \\
b(r,g)&=&-\frac{c_1^2}{r^3}+\frac{m_2\sqrt{3}}{3 r^2}+\frac{3 (k-1) m_2 r}{\left(r^2+k^2 g(r)^2\right)^{3/2}} \\
c(r,g)&=& -\frac{3 k m_2 g(r)}{\left(r^2+k^2 g(r)^2\right)^{3/2}} 
\end{eqnarray*}

\end{thm}

\section{The n-Body Problem}

The same type of results generalize for the n-Body problem. Let us see this.


\begin{thm}\label{t7}
The functions \begin{eqnarray*}
x(t)&=&\left(0,0,f(t)\right)\\
y^1(t)&=&\left(r(t)\cos(\theta(t))\, ,\, r(t)\sin(\theta(t))\, ,\, -\frac{m_1}{n m_2}\, f(t)\right)\\
y^2(t)&=&\left(r(t)\cos(\theta(t)+\frac{2 \pi}{n})\, ,\, r(t)\sin(\theta(t)+\frac{2 \pi}{n})\, ,\, -\frac{m_1}{ n m_2}\, f(t)\right)\\
  &\vdots& \\
y^n(t)&=&\left(r(t)\cos(\theta(t)+\frac{2(n-1) \pi}{n})\, ,\, r(t)\sin(\theta(t)+\frac{2 (n-1)\pi}{n})\, ,\, -\frac{m_1}{nm_2}\, f(t)\right)\\
\end{eqnarray*}

provides a solution of the $(n+1)-$body problem with $n>1$, if and only if

\begin{eqnarray}\label{e2}
\ddot{f}&=&-\frac{(m_1+nm_2)f}{h^3}\\
 \ddot{r} &=& \frac{c_1^2}{r^3}-a_n \frac{m_2}{r^2}-\frac{m_1r}{h^3} \quad\hbox{and}\quad r^2\dot{\theta}=c_1
\end{eqnarray}

where  $h=\sqrt{r^2+(\frac{m_1+n m_2}{n m_2})^2f^2}$,  $c_1=y_1(0)\, \dot{y_2}(0)-y_2(0)\, \dot{y_1}(0)$ and 

\begin{eqnarray}
a_n=\sum_{j=1}^{n-1} \frac{1-\hbox{e}^{\frac{2 \pi j i}{n}}}{|\hbox{e}^{\frac{2 \pi j i}{n}}-1|^3}
\end{eqnarray}

\end{thm}

\begin{proof}
This is a direct computation similar to those made on the proof of Theorem (\ref{t1}). 
\end{proof}

A direct computation shows that $a_2=\frac{1}{4},\,  a_3=\frac{1}{\sqrt{3}}$ and $ a_4=\frac{1}{4}+\frac{1}{\sqrt{2}}$

\begin{cor}
When the function $f$ is identically zero, the body with mass $m_1$ will not move and the other two bodies will move on a conic.  For the bounded solutions the bodies move along ellipses and are periodic. In this case we have {\it explicit} formulas for the solutions in the same way that we do for the two body problem. \end{cor}

\begin{exmp} As an example here we show an explicit example of the $19$-Body Problem. If we take $m_1=2$

\begin{eqnarray*}m_2&=&24\, \Big{(}4 \sqrt{3}+15+6 \sqrt{\frac{2}{1+\cos \left(\frac{\pi }{9}\right)}}+6 \sqrt{\frac{2}{1+\cos \left(\frac{2 \pi }{9}\right)}}+ 6 \csc \left(\frac{\pi }{9}\right)+\\
& &6 \csc \left(\frac{\pi }{18}\right)+6 \sqrt{\frac{2 \csc \left(\frac{\pi }{18}\right)}{\csc \left(\frac{\pi }{18}\right)-1}}+6 \csc \left(\frac{\pi }{18}\right) \sqrt{\frac{2 \sin \left(\frac{\pi }{18}\right)}{1+\csc \left(\frac{\pi }{18}\right)}}\quad \Big{)}^{-1} \approx 0.231508,\, 
\end{eqnarray*}

then, a solution of the 19-Body problem with masses: $m_1$ (described by $x(t)$) and $m_2$ for the other 18 bodies (described by $y^j(t)$, $j=1,\dots,18$) is given by

$$x(t)=(0,0,0)\, \hbox{and} \, y^j(t)=\left(r(t)\cos(\theta(t)+\frac{2(j-1) \pi}{18})\, ,\, r(t)\sin(\theta(t)+\frac{2 (j-1)\pi}{18}\, ,\, 0\right)$$

with $r(t)=\frac{1}{1-\frac{1}{2}\, \cos(\theta(t))}$ and $\theta(t)$ a solution of $r(t)^2\dot{\theta}(t)=-2$.

\end{exmp}

\begin{figure}[hbtp]
\begin{center}
\includegraphics[width=.4\textwidth]{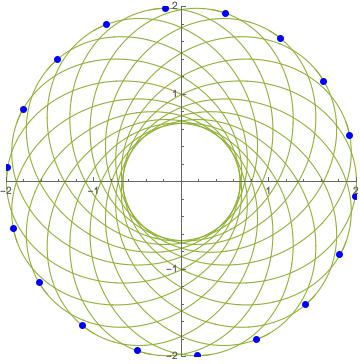}
\end{center}
\caption{Explicit solutions can be easily found for the solutions of the $n$-body that we are considering  when $f$ is identically zero, the ordinary differential equation in this case reduces to the ordinary differential equation  for the two body problem.  }
\end{figure}

\begin{thm}\label{t8}
If a solution of the (n+1)-Body problem considered in Theorem \ref{t7} has nonzero angular momentum, this is $c_1\ne0$, and it also has negative total energy, then, there is no collision. Moreover

\begin{eqnarray} \label{dDelta}
D=2 n c_1^2 c_2 m_2+n^2\, ( m_1m_2+\frac{b_n}{2} m_2^2)^2>0
\end{eqnarray}

and 

\begin{eqnarray} \label{rb}
\frac{-nm_1m_2-\frac{nb_n}{2} m_2^2+\sqrt{D}}{2c_2}<r(t)<\frac{-nm_1m_2-\frac{nb_n}{2} m_2^2-\sqrt{D}}{2c_2}
\end{eqnarray}

where

\begin{eqnarray} \label{dDelta}
b_n=\sum_{j=1}^{n-1} |\hbox{e}^{\frac{2 \pi j i}{n}}-1|^{-1}
\end{eqnarray}

\end{thm}
\begin{proof}
This is a direct computation similar to those made on the proof of Theorem (\ref{t2}). 
\end{proof}

A direct computation shows that $b_2=\frac{1}{2}$, $b_3=\frac{2}{\sqrt{3}}$, $b_4=\frac{1}{2}+\sqrt{2}$ and $b_5=2\sqrt{1+\frac{2}{\sqrt{5}}}  $

\begin{thm}\label{bsnBP}
The solutions of the $(n+1)-$Body Problem considered in Theorem \ref{t7} with $c_1\ne0$, $c_2<0$ and $2 c_1^2c_2+n a_nm_2^2(2m_1+(b_n-a_n)m_2) <0$ are bounded.
\end{thm}
\begin{proof}
This is a direct computation similar to that made on the proof of Theorem (\ref{bs3BP}). 
\end{proof}

\begin{cor}
There are infinitely many periodic solutions considered in Theorem \ref{t4}
\end{cor}

\begin{thm}\label{t6}
For every solution of the 3-Body problem considered in Theorem \ref{t1}, we have that 
\begin{eqnarray}\label{odefg4BP}
a(r,g,g^\prime)\,  g^{\prime\prime}(r)+b(r,g)\, g^\prime+c(g,r)=0
\end{eqnarray}

where $k=\frac{m_1+nm_2}{nm_2} $ and

\begin{eqnarray*}
a(r,g,g^\prime)&=&-\frac{2 n^2 m_1 m_2^2 r^2+\left( n^2b_nm_2^3r + 2nc_2m_2r^2-n^2c_1^2m_2^2\right) \sqrt{r^2+k^2 g(r)^2}}{ r^2 \sqrt{r^2+k^2 g(r)^2} \left(n^2 m_2^2+m_1(m_1+nm_2) g'(r)^2\right)}  \\
b(r,g)&=&-\frac{c_1^2}{r^3}+\frac{m_2a_n}{ r^2}+\frac{ n (k-1) m_2 r}{\left(r^2+k^2 g(r)^2\right)^{3/2}} \\
c(r,g)&=& -\frac{n k m_2 g(r)}{\left(r^2+k^2 g(r)^2\right)^{3/2}} 
\end{eqnarray*}

\end{thm}

\section{Examples of periodic solutions}

For the $(n+1)-$Boby Problem, the bounded solutions are among those that satisfy $c_1\ne0$, $c_2<0$ and the additional condition $2 c_1^2c_2+n a_nm_2^2(2m_1+(b_n-a_n)m_2) <0$. We also have that the solution is periodic if the function 

\begin{eqnarray*}\label{xi}
\xi(y_{10},dy_{20},df_0,m_1,m_2,t_0)= (f^\prime(t_0)-df_0)^2+(\theta(t_0)-\theta_0)^2+r^\prime(t_0)^2+f(t_0)^2 
\end{eqnarray*}

vanishes - please refer to Definition \ref{xi} for the particular case of the 3-Body Problem- The examples provided here satisfy that the function $\xi(q_0)<10^{-11}$, and since we are considering masses and initial condition of the order of 1, 10 and 100, this solution have the property that the initial conditions agrees with the values of the solution at $t=t_0$ with 5, 6 or 7 significant digits.

\subsection{Examples in the Three-body Problem}

These initial conditions make the function $\xi$ smaller than $10^{-11}$
\begin{enumerate}
\item $y_{10}=11.3361$, $dy_{20}=2.20041$, $df_0=1.5009$, $t_0=18.5318$, $m_1=41.0495$ and $m_2=81.3134$. For this example the angle is $\frac{7\pi}{6}$
\item $y_{10}=17.6054$, $dy_{20}=1.24777$, $df_0=2.20234$, $t_0=7.01051$, $m_1=459.197$ and $m_2=663.716$. For this example the angle is $\frac{23\pi}{12}$
\item $y_{10}=9.63668$, $dy_{20}=3.27637$, $df_0=5.48683$, $t_0=12.3248$, $m_1=107.622$ and $m_2=236.949$. For this example the angle is $\frac{3\pi}{2}$
\item $y_{10}=10.4734$, $dy_{20}=3.42417$, $df_0=3.92526$, $t_0=11.4716$, $m_1=100.583$ and $m_2=227.047$. For this example the angle is $\frac{4\pi}{3}$
\item $y_{10}=9.01442$, $dy_{20}=3.22148$, $df_0=6.13575$, $t_0=12.3114$, $m_1=112.041$ and $m_2=237.422$. For this example the angle is $\frac{5\pi}{3}$
\item $y_{10}=9.01442$, $dy_{20}=3.22148$, $df_0=6.13575$, $t_0=12.3114$, $m_1=112.041$ and $m_2=237.422$. For this example the angle is $\frac{5\pi}{3}$
\item $y_{10}=9.41501$, $dy_{20}=3.1106$, $df_0=6.20044$, $t_0=13.5238$, $m_1=117.811$ and $m_2=240.987$. For this example the angle is $\frac{7\pi}{4}$
\item $y_{10}=9.86766$, $dy_{20}=3.29664$, $df_0=5.10437$, $t_0=12.1545$, $m_1=102.817$ and $m_2=233.913$. For this example the angle is $\frac{17\pi}{12}$
\item $y_{10}=9.65614$, $dy_{20}=3.19928$, $df_0=5.7356$, $t_0=12.987$, $m_1=110.723$ and $m_2=238.196$. For this example the angle is $\frac{19\pi}{12}$

\item $y_{10}=10.7714$, $dy_{20}=3.5243$, $df_0=3.1821$, $t_0=10.9141$, $m_1=98.7543$ and $m_2=230.66$. For this example the angle is $\frac{7\pi}{6}$
\item $y_{10}=11.9793$, $dy_{20}=3.4739$, $df_0=2.02892$, $t_0=11.629$, $m_1=94.6866$ and $m_2=237.265$. For this example the angle is $\frac{13\pi}{12}$

\item $y_{10}=11.3215$, $dy_{20}=1.82051$, $df_0=3.75072$, $t_0=30.5561$, $m_1=242.104$ and $m_2=92.3726$. For this example the angle is $\frac{21\pi}{4}$
\end{enumerate}

\subsection{Examples in the Four-body Problem}

These initial conditions make the function $\xi$ smaller than $10^{-11}$
\begin{enumerate}

\item $y_{10}=11.3215$, $dy_{20}=1.82051$, $df_0=3.7507$, $t_0=30.5561$, $m_1=242.104$ and $m_2=92.3721$. For this example the angle is $\frac{21\pi}{4}$
\item $y_{10}=11.4248$, $dy_{20}=1.618332$, $df_0=3.7506$, $t_0=30.64$, $m_1=238.692$ and $m_2=92.357$. For this example the angle is $\frac{11\pi}{2}$

\item $y_{10}=11.7799$, $dy_{20}=0.9102$, $df_0=3.64$, $t_0=30.4953$, $m_1=216.145$ and $m_2=90.1003$. For this example the angle is $\frac{13\pi}{2}$

\item $y_{10}=11.5328$, $dy_{20}=1.43002$, $df_0=3.7398$, $t_0=30.7416$, $m_1=233.381$ and $m_2=92.1394$. For this example the angle is $\frac{23\pi}{4}$

\item $y_{10}=11.6333$, $dy_{20}=1.26332$, $df_0=3.7399$, $t_0=30.9651$, $m_1=244.691$ and $m_2=92.0296$. For this example the angle is $6 \pi$


\end{enumerate}

\end{document}